\documentclass[11pt,,leqno,twoside]{article} 
\usepackage{bbm}
\usepackage{mathrsfs}
\usepackage{amssymb}
\usepackage{amsmath}
\usepackage{amsthm}
\usepackage{amsfonts}
\usepackage{color}
\usepackage{graphicx}

\usepackage{mathrsfs}
\usepackage{graphicx}

\allowdisplaybreaks

\usepackage{titletoc}
\titlecontents{section}[0pt]{\addvspace{2pt}\filright}
              {\contentspush{\thecontentslabel\ }}
              {}{\titlerule*[8pt]{.}\contentspage}

\def\rr{{\mathbb R}}
\def\rn{{{\rr}^n}}

\def\nn{{\mathbb N}}

\def\ca{{\mathcal A}}

\def\fz{\infty}
\def\az{\alpha}

\def\lip{{\mathop\mathrm{\,Lip}}}

\def\lz{\lambda}
\def\dz{\delta}

\def\eps{\epsilon}
\def\ez{\epsilon}

\def\bz{\beta}

\def\gz{{\gamma}}

\def\boz{{\Omega}}

\def\wz{\widetilde}

\def\bint{{\ifinner\rlap{\bf\kern.35em--}
\int\else\rlap{\bf\kern.45em--}\int\fi}\ignorespaces}

\def\bbint{{\ifinner\rlap{\bf\kern.35em--}
\hspace{0.078cm}\int\else\rlap{\bf\kern.45em--}\int\fi}\ignorespaces}

\def\esup{\mathop\mathrm{\,esssup\,}}

\def\r{\right}
\def\lf{\left}

\newtheorem{thm}{Theorem}[section]
\newtheorem{lem}[thm]{Lemma}
\newtheorem{defn}[thm]{Definition}
\numberwithin{equation}{section}

\begin{document}

\arraycolsep=1pt

\allowdisplaybreaks

\title{\Large\bf
Everywhere differentiability of viscosity solutions to a class of Aronsson's equations 
\footnotetext{\hspace{-0.35cm}
\noindent  {\it Key words and {phrases}:}    $L^\fz$-variational problem, absolute minimizer, everywhere differentiability,
Aronsson's equation.}
\author{Juhana Siljander, Changyou Wang, and Yuan Zhou}
\date{\today}}
\maketitle

\begin{center}
\begin{minipage}{13.5cm}\small
{\noindent{\bf Abstract}\quad
For any open set $\Omega\subset\rn$ and $n\ge 2$,
we establish everywhere differentiability of viscosity solutions to the Aronsson equation
$$ \langle D_x(H(x,  Du)), D_p H(x,  Du)\rangle=0 \quad \rm in\ \ \boz,
$$
where $H$ is given by
$$H(x,\,p)=\langle A(x)p,p\rangle=\sum_{i,\,j=1}^na^{ij}(x)p_i p_j,\ x\in\Omega, \ p\in\mathbb R^n, $$
and $A=(a^{ij}(x))\in C^{1,1}(\overline\Omega,\mathbb R^{n\times n})$
is uniformly elliptic. This extends an earlier
theorem by Evans and Smart \cite{es11a} on infinity harmonic functions.
}
\end{minipage}
\end{center}

\section{Introduction}

For any open set $\Omega\subset\mathbb R^n$ with $n\ge 2$,
we consider the Aronsson equation:
\begin{equation}\label{e1.x1}
\mathcal A_{H}[u](x):=\langle D_x(H(x,  Du(x))), D_p H(x,  Du(x))\rangle=0 \quad \rm in\ \ \ \boz,
\end{equation}
where $H$ is given by
\begin{equation}\label{hamilton}
H(x,\,p)=\langle A(x)p,p\rangle=\sum_{i,\,j=1}^na^{ij}(x)p_ip_j, \ x\in\Omega
\ {\rm{and}}\ \ p\in\mathbb R^n,
\end{equation}
and the coefficient matrix $A=(a^{ij}(x))_{1\le i,j\le n}$
is uniformly elliptic: $\exists$ $L>0$ such that
\begin{equation}\label{ellipticity}
{L}^{-1}|p|^2\le \langle A(x)p,\,p\rangle\le L|p|^2,\  x\in\boz\ {\rm{and}}\ p\in\rr^n.
\end{equation}
The set of uniformly elliptic coefficient matrices is denoted as $\mathscr A(\Omega)$.

Our main interest concerns the regularity issue of viscosity solutions of the Aronsson equation (\ref{e1.x1}).
In this context, we are able to extend an important result of  Evans and Smart \cite{es11a} on infinity harmonic functions by
proving
\begin{thm}\label{t1.1}
Assume $A\in\mathscr A(\Omega)\cap C^{1,1}( \boz)$. Then any viscosity solution $u\in C(\overline\Omega)$ to the Aronsson equation
\eqref{e1.x1} is everywhere differentiable in $\boz$.
\end{thm}

Note that when $A$ is the identity matrix of order $n$, the Aronsson equation (\ref{e1.x1}) becomes the infinity Laplace equation:
\begin{equation}\label{inf_lap}
\Delta_\infty u :=\sum_{i,j=1}^n u_{x_i}u_{x_j}u_{x_ix_j}=0 \quad {\rm{in}}\quad \Omega.
\end{equation}
G. Aronsson \cite{a1,a2,a3,a4} initiated the study of the infinity Laplace equation (\ref{inf_lap}) by deriving it as the Euler-Lagrange equation,
in the context of $L^\infty$-variational problems, of absolute minimal Lipschitz extensions (AMLE)
or equivalently absolute minimizers (AM) of
\begin{equation}\label{AMLE}
\inf\Big\{\esup_{x\in \Omega} |Du|^2: \ u\in {\rm{Lip}}(\Omega)\Big\}.
\end{equation}
Employing the theory of viscosity solutions of elliptic equations,  Jensen~\cite{j1993} has first proved the equivalence
between AMLEs and viscosity solutions of (\ref{inf_lap}), and the uniqueness of both AMLEs and infinity harmonic functions
under the Dirichlet boundary condition. See \cite{pssw} and \cite{as} for alternative proofs.
For further properties of infinity harmonic functions, we refer the readers to the paper by Crandall-Evans-Gariepy \cite{ceg} and
the survey articles by Aronsson-Crandall-Juutinen~\cite{acj} and Crandall~\cite{c2008}.

For $L^\infty$-variational problems involving Hamiltonian functions
$H=H(x,z,p)\in C^2(\Omega\times\mathbb R\times\mathbb R^n)$, Barron, Jensen and Wang~\cite{bjw} have proved that  an absolute minimizer of
\begin{equation}\label{e1.x2}
\mathscr F_\infty(u,\Omega)=\esup_{x\in \Omega} H(x, u(x), Du(x))
\end{equation}
is a viscosity solution of \eqref{e1.x1}, provided $H$ is level set convex in $p$-variable.
Recall that a Lipschitz function $u\in \lip(\boz)$ is an
{\it absolute  minimizer}  for $\mathscr F_\infty$, if for every open subset
$U\Subset \boz$  and $v\in \lip(U)$, with $v|_{\partial U}= u|_{\partial U}$, it holds
$$\mathscr F_\infty(u,U) \le \mathscr F_\infty(v,U).$$
See \cite{cwy}, \cite{acjs}, \cite{jwy}, and \cite{ksz} for related works on
both Aronsson's equations (\ref{e1.x1}) and absolute minimizers of $\mathscr F_\infty$.

The issue of regularity of infinity harmonic functions (or viscosity solutions to \eqref{inf_lap})
has attracted great interests.
When $n=2$,  Savin \cite{s05} showed the interior $C^1$-regularity,
and  Evans-Savin \cite{es08} established the interior $C^{1,\,\alpha}$-regularity.
Wang and Yu \cite{wy1} have established the $C^1$-boundary  regularity.
Wang and Yu \cite{wy} have also extended Savin's $C^1$-regularity to
the Aronsson equation (\ref{e1.x1}) for uniformly convex $H(p)\in C^2(\mathbb R^2)$.
When $n\ge 3$, Evans-Smart \cite{es11a,es11b} have established the interior everywhere differentiability
of infinity harmonic functions, Wang-Yu \cite{wy1} have proved the boundary differentiability of infinity harmonic functions, 
and Lindgren \cite{l} has shown the everywhere differentiability
for inhomogeneous infinity Laplace equation.

In this paper, we are able to prove Theorem 1.1 by extending the techniques  by Evans-Smart \cite{es11a,es11b} to the Aronsson equation 
(\ref{e1.x1}) for $A \in\mathscr A(\Omega)\cap C^{1,1}(\Omega)$ and $n \ge 2$.  It is an interesting question to ask whether
Theorem 1.1 holds for $A\in\mathscr A(\Omega)\cap C^1(\Omega)$.



\section{Preliminaries}

In this section, we will describe a regularization scheme of the Aronsson equation (\ref{e1.x1}).
First, let's recall the definition of viscosity solutions of the Aronsson equation~\eqref{e1.x1}.

\begin{defn}
A function $u \in C(\overline\Omega)$ is a viscosity subsolution (supersolution) of the Aronsson equation (\ref{e1.x1})
if, for every $x \in \Omega$ and every $\varphi \in C^2(\Omega)$ such that if $u-\varphi$ has a local maximum (minimum) at $x$ then
\begin{equation}\label{e2.x1}
\mathcal A_{H} [\varphi](x) \  \ge (\le) \ 0.
\end{equation}
A function $u$ is a viscosity solution of (\ref{e1.x1}) if $u$ is both viscosity subsolution and supersolution.
\end{defn}

For $\epsilon>0$ and a uniformly elliptic matrix $B\in \mathscr A(\Omega)\cap C^\infty(\Omega)$,
set the Hamiltonian function $H_B$ by
$$H_B(x,p)=\langle B(x) p, p\rangle,
\ x\in\Omega \ {\rm{and}}\ p\in\mathbb R^n.$$
We consider an $\epsilon$-regularized Aronsson equation (\ref{e1.x1}) associated
with $B$ and $H_B$:
 \begin{equation}\label{e2.x2}
\begin{cases}
-\mathcal A_{H_B}^{\epsilon}[u^\epsilon]:=-\mathcal A_{H_B}[u^\eps] -\eps {\rm{div}}(B\nabla u ^{\ez}) = 0 &\quad \text{in}\quad \boz ,\\
\ \qquad\qquad\qquad\qquad\qquad\qquad\qquad \ \ \ \ u ^\ez = u &\quad \text{on}\quad \partial \boz.
\end{cases}
\end{equation}
For (\ref{e2.x2}), we  have the following theorem.
\begin{thm}\label{existence} For $\epsilon>0$, $B\in \mathscr A(\Omega)\cap C^{\infty}(\Omega)$, and $u\in C^{0,1}(\Omega)$, there exists a unique solution $u^\epsilon\in C^{\infty}(\Omega)\cap C(\overline\Omega)$ of the equation (\ref{e2.x2}).
\end{thm}
\begin{proof} Consider the minimization problem of the functional of exponential growth
$$c_\epsilon:=\inf\Big\{\mathcal I_\epsilon[v]:=\int_\Omega \exp\big(\frac{1}{\epsilon}H_B(x,\nabla v)\big)\,dx\ \big| \  v\in {\bf K}_\epsilon\Big\},$$
where ${\bf K}_\epsilon$ is the set of admissible functions of the functional $\mathcal I_\epsilon$
defined by
$${\bf K}_\epsilon=\Big\{ w\in W^{1,1}(\Omega) \big|
\int_\Omega \exp\big(\frac{1}{\epsilon}H_B(x,\nabla w)\big)\,dx<+\infty,
\ w=u \ {\rm{on}}\ \partial\Omega\Big\}.
$$
Note that since $u\in {\bf K}_\epsilon$, ${\bf K}_\epsilon\not=\emptyset$.  Let $\{u_m\}\subset {\bf K}_\epsilon$ be a minimizing sequence, i.e., $\displaystyle\lim_{m\rightarrow\infty}\mathcal I_\epsilon[u_m]=c_\epsilon$. Without loss of generality,
we may assume that there exists $u^\epsilon\in {\bf K}_\epsilon$ such that
$u_m\rightarrow u^\epsilon$ uniformly on $\Omega$, and
$Du_m\rightharpoonup Du^\epsilon$ in $L^q(\Omega)$ for any $1\le q<+\infty$.
Since $H_B(x,p)=\langle B(x) p, p\rangle$ is uniformly convex in $p$-variable, by the lower semicontinuity we have that
\begin{eqnarray*}
\mathcal I_\epsilon[u^\epsilon]&=&\int_\Omega \exp\big(\frac{1}{\epsilon}H_B(x,\nabla u^\epsilon)\big)\,dx=\sum_{k=0}^\infty \int_\Omega \frac{\big(\epsilon^{-1}H_B(x,\nabla u^\epsilon))^k}{k!}\,dx\\
&\le&\liminf_{m\rightarrow\infty}\sum_{k=0}^\infty \int_\Omega \frac{\big(\epsilon^{-1}H_B(x,\nabla u_m))^k}{k!}\,dx\\
&=&\liminf_{m\rightarrow\infty}\int_\Omega \exp\big(\frac{1}{\epsilon}H_B(x,\nabla u_m)\big)\,dx
=\liminf_{m\rightarrow\infty} \mathcal I_\epsilon[u_m]=c_\epsilon.
\end{eqnarray*}
Hence $c_\epsilon= \mathcal I_\epsilon[u^\epsilon]$ and $u^\epsilon$ is a minimizer of
$\mathcal I_\epsilon$ over the set ${\bf K}_\epsilon$.
Direct calculations imply that the Euler-Lagrange equation of $u^\epsilon$ is (\ref{e2.x2}). The uniqueness of $u^\epsilon$ follows
from the maximum principle that is applicable of (\ref{e2.x2}). The smoothness of $u^\epsilon$
follows from the theory of quasilinear uniformly elliptic equations, and the reader can find
its proofs in the papers by Lieberman \cite{l1} page 47-49 and \cite{l2} lemma 1.1 (see also
the paper by Duc-Eells \cite{de}).
\end{proof}

Note that any viscosity solution $u\in C(\overline\Omega)$ of the Aronsson equation
(\ref{e1.x1}) is locally Lipschitz continuous, i.e. $u\in C^{0,1}_{\rm{loc}}(\Omega)$
(see \cite{bfm} and \cite{ksz}). Since we consider the interior regularity of $u$,
we may simply assume that $u\in C^{0,1}(\Omega)$.

Now we will indicate that under suitable conditions on $A$,
any viscosity solution $u\in C^{0,1}(\Omega)$ of the Aronsson equation (\ref{e1.x1}) can be approximated by smooth solutions $u^\epsilon$ of $\epsilon$-regularized
equations (\ref{e2.x2}) associated with suitable $H_B$'s. For this, we recall that
for any $A\in \mathscr A(\Omega)\cap C^{1,1}(\Omega)$, it is a standard fact that
there exists $\{A_\epsilon\}\subset \mathscr A(\Omega)\cap C^\infty(\Omega)$ such that
\begin{itemize}
\item[(2.1)] $\big\|A_\epsilon\big\|_{C^{1,1}(\Omega)}\le 2\big\|A\big\|_{C^{1,1}(\Omega)}$
for all $\epsilon>0$.
\item [(2.2)] For any $\alpha\in (0,1)$, $A_\epsilon\rightarrow A$ in $C^{1,\alpha}(\Omega)$
as $\epsilon\rightarrow 0$.
\end{itemize}

\begin{thm}\label{approximation} For  any $A\in \mathscr A(\Omega)\cap C^{1,1}(\Omega)$ with ellipticity
constant $L<2^{\frac15}$ (see (\ref{ellipticity})), let $\{A_\epsilon\}\subset \mathscr A(\Omega)\cap C^\infty(\Omega)$ satisfy the properties
(2.1) and (2.2). 
Assume that $u\in C^{0,1}(\Omega)$ is a viscosity solution of the Aronsson equation
(\ref{e1.x1}), and $\{u^\epsilon \}\subset C^\infty(\Omega)\cap C(\overline \Omega)$ are classical solutions
of the $\epsilon$-regularized equation (\ref{e2.x2}) on $\Omega$, with $B$ and $H_B$ replaced by $A_\epsilon$
and $H_{A_\epsilon}$ respectively.
Then there
exists a constant $\delta_0=\dz_0(\boz, \|A\|_{L^\infty(\Omega)})>0$ such that if $\| DA\|_{L^\infty( \boz)}\le \dz_0$,
 then $u^\epsilon\rightarrow u$ in $C^0_{\rm{loc}}(\Omega)$.
\end{thm}
\begin{proof} From Theorem 3.1, we have that for any compact subset $K\Subset\Omega$,
\begin{eqnarray*}
\big\|Du^\epsilon\big\|_{C(K)}&\le& C\Big({\rm{dist}}(K, \partial\Omega), \|u\|_{C(\overline\Omega)}, \|A_\epsilon\|_{C^{1,1}(\Omega)}\Big)\\
&\le& C\Big({\rm{dist}}(K, \partial\Omega), \|u\|_{C(\overline\Omega)}, \|A\|_{C^{1,1}(\Omega)}\Big),
\ \forall\ \epsilon>0.
\end{eqnarray*} 
This implies that there exists a $\hat u\in C^{0,1}_{\rm{loc}}(\Omega)$ such that, after passing to a subsequence,
\begin{equation}\label{loc_uniform}
u_\epsilon\rightarrow \hat u \ {\rm{in}}\ C^0_{\rm{loc}}(\Omega).
\end{equation}
Since $\{A_\epsilon\}$ satisfies (2.1) and (2.2), there exists $\epsilon_0>0$ such that for any $0<\epsilon\le\epsilon_0$, it holds
that $\displaystyle\|A_\epsilon\|_{L^\infty(\Omega)}\le 2\|A\|_{L^\infty(\Omega)},$
and the ellipticity constant $L_\epsilon$ of $A_\epsilon$ satisfies $L_\epsilon\le 2^\frac14$.
Let $\dz_0>0$ be the constant given by Theorem 3.2 and
assume $\displaystyle \|DA\|_{L^\infty(\Omega)}\le\frac{\delta_0}2$. Then there exists $0<\epsilon_1\le\epsilon_0$
such that $\displaystyle \|DA_\epsilon\|_{L^\infty(\Omega)}\le\delta_0$ for any $\epsilon<\epsilon_1$.
Thus Theorem 3.2 below is applicable to $u_\epsilon$ for any $0<\epsilon<\epsilon_1$ and we conclude
that there exist $\gamma\in (0,1)$ and $C>0$, independent of $0<\epsilon<\epsilon_1$, such that
\begin{equation}\label{bdry_uniform}
\big|u_\epsilon(x)-u(x_0)\big|\le C|x-x_0|^\gamma, \ \forall\ x\in\Omega, \ x_0\in\partial\Omega.
\end{equation}
From (\ref{loc_uniform}) and (\ref{bdry_uniform}), we see that
$$|\hat{u}(x)-u(x_0)|\le C|x-x_0|^\gamma, \ \forall\ x\in\Omega, \ x_0\in\partial\Omega.$$
This implies that $\hat{u}\in C(\overline\Omega)$ and $\hat u\equiv u $ on $\partial\Omega$.
By the compactness property of viscosity solutions of elliptic equations (see Crandall-Ishii-Lions \cite{CIL}), we know that $\hat{u}
\in C(\overline\Omega)$ is a viscosity solution of the Aronsson equation (\ref{e1.x1}) associated with
$A$ and $H_A$. Since $\hat{u}\equiv u$ on $\partial\Omega$, it follows from the uniqueness
theorem of (\ref{e1.x1}) (see \cite{bfm} and \cite{ksz}) that $\hat{u}=u$. This also implies that
$u^\epsilon\rightarrow u$ in $C^0_{\rm{loc}}(\Omega)$ for $\epsilon\rightarrow 0$.
\end{proof}

\section{A priori estimates}

Motivated by \cite{es11a, es11b}, we will establish some necessary {a priori} estimates
of smooth solutions $u^\epsilon$ of the equation (\ref{e2.x2}) associated with $A_\epsilon$
satisfying (2.1) and (2.2), which is the crucial ingredient to establish everywhere differentiability
of viscosity solution of the Aronsson equation (\ref{e1.x1}).

In this section, we will assume $A\in\mathscr A(\Omega)\cap C^\infty(\Omega)$,
and $u^\ez\in C^{\infty}(\Omega)\cap C(\overline\Omega)$ is a solution
of the $\epsilon$-regularized equation \eqref{e2.x2} with $B$ and
$H_B$ replaced by $A$ and $H_A$.

\subsection{Lipschitz estimates}
We begin with the following theorems.
\begin{thm}\label{first_lemma} For $u\in C^{0,1}(\Omega)$ and
$A\in\mathscr A(\Omega)\cap C^\infty(\Omega)$,
assume $u^\ez\in C^{\infty}(\Omega)\cap C(\overline\Omega)$ is a solution
of the $\epsilon$-regularized equation \eqref{e2.x2}, with $B$ and
$H_B$ replaced by $A$ and $H_A$. Then we have the estimates
\begin{equation}\label{max_est}
\max_{\overline{\Omega}} |u^{\eps}| \le \max_{\overline\Omega} |u|,
\end{equation}
and for each open set $V \Subset \Omega$, there exists $C>0$ depending on $n, {L}, \|u\|_{C(\overline\Omega)}, {\rm{dist}}(V, \partial\Omega),$ and $\|A\|_{C^{1,1}(\Omega)}$ such
that
\begin{equation}\label{gradient_bound}
\max_{\overline{V}} |D u^\eps| \le C.
\end{equation}
\end{thm}
\begin{proof}
The estimate (\ref{max_est}) follows from the standard maximum principle of the equation \eqref{e2.x2}.
For (\ref{gradient_bound}), we proceed as follows.
To simplify the presentation,  we will use the Einstein summation convention.
Denote
$u^\eps_{i }=\frac{\partial}{\partial x_i}u^\ez$, $u^\eps_{ij }=\frac{\partial^2}{\partial x_i\partial x_j}u^\ez$,
$a^{ij}$ as the $(i,j)^{\rm{th}}$-entry of $A$, and $a_k^{ij}= \frac{\partial}{\partial x_k}a^{ij}$.
Recall that
$$\ca_H [u^\eps]=  2a^{ik}u_k^\ez  u^\eps_{ij} a^{j\ell}u_\ell^\ez+ a_k^{ij}u^\ez_iu^\ez_j a^{k\ell}u_\ell^\ez.$$
Taking $\frac{\partial}{\partial s}$ of the equation \eqref{e2.x2}, we obtain
\begin{eqnarray}\label{diff_aron}
&&2a^{ik}u_k^\ez  u^\eps_{ijs} a^{j\ell}u_\ell^\ez+ 4a_s^{ik}u_k^\ez  u^\eps_{ij} a^{j\ell}u_\ell^\ez+
4a^{ik}u_{ks}^\ez  u^\eps_{ij} a^{j\ell}u_\ell^\ez
+ a_{ks}^{ij}u^\ez_iu^\ez_j a^{k\ell}u_\ell^\ez+2a_k^{ij}u^\ez_{is}u^\ez_j a^{k\ell}u_\ell^\ez\nonumber\\
&&\qquad\qquad\qquad\ \ \ +a_k^{ij}u^\ez_iu^\ez_j a_s^{k\ell}u_\ell^\ez
+a_k^{ij}u^\ez_iu^\ez_j a^{k\ell}u_{\ell s}^\ez
+\eps {\,\rm div} (A D u^{\ez}_s)+\ez{\,\rm div} (A_s D u^{\ez} ) =0.
\end{eqnarray}
Set
\begin{equation}\label{G-ep}
G_{m}^\epsilon:=4a^{im}    u^\eps_{ij} a^{j\ell}u_\ell^\ez
+2a_k^{mj} u^\ez_j a^{k\ell}u_\ell^\ez+a_k^{ij}u^\ez_iu^\ez_j a^{km},
\end{equation}
and
\begin{equation}\label{F-ep}
 F_s^\ez:= 4a_s^{ik}u_k^\ez  u^\eps_{ij} a^{j\ell}u_\ell^\ez+a_k^{ij}u^\ez_iu^\ez_j a_s^{k\ell}u_\ell^\ez+a_{ks}^{ij}u^\ez_iu^\ez_j a^{k\ell}u_\ell^\ez+\eps {\,\rm div} (A_s D u^{\ez}).
\end{equation}
Define the operator $L_\eps$ by
\begin{equation}\label{L-ep}
 L _\ez v:= 2a^{ik}u_k^\ez  v_{ij} a^{j\ell}u_\ell^\ez+ \sum_{m=1}^nG_m^\epsilon v_m
+\eps {\,\rm div} (A D v).
\end{equation}
Then (\ref{diff_aron}) can be written as
 \begin{equation}\label{diff_aron1}
-L _\ez (u^\ez_s)= F_s^\epsilon.
\end{equation}
Set $v^\ez:=\frac12|D u^\ez |^2. $
Then
$$v^\epsilon_i= \sum_{s=1}^n u^\epsilon_s u^\epsilon_{si}\ {\rm{and}}\  v^\ez_{ij}= \sum_{s=1}^n \big[u^\ez_{si}u^\ez_{sj}+  u^\ez_{sij}u^\ez_{s}\big],  $$
so that by using the equation (\ref{diff_aron1}) we have
 \begin{align}\label{diff_aron2}
L _\ez v^\ez&= \sum_{s=1}^n\big[ 2a^{ik}u_k^\ez  u^\ez_{si}u^\ez_{sj} a^{j\ell}u_\ell^\ez
+u_s^\epsilon L_\epsilon u_s^\epsilon+ \eps   a^{ij}u^\ez_{si}u^\ez_{sj} \big]
   \nonumber\\
&= 2 |D^2u^\ez ADu^\ez|^2+ \sum_{s=1}^n\big[\eps a^{ij}u^\ez_{si}u^\ez_{sj}
-u^\ez_s F_s^\ez \big ].
\end{align}
Set $z^\ez:=\frac12(u^\ez)^2$. Then by the equation (\ref{e2.x2})  we have
\begin{eqnarray}L _\ez z^\ez&&= 2a^{ik}u_k^\ez  u^\ez_{ ij} u^\ez a^{j\ell}u_\ell^\ez+2a^{ik}u_k^\ez  u^\ez_{ i } u^\ez_j a^{j\ell}u_\ell^\ez\nonumber
 +\sum_{m=1}^nG_m^\ez u^\ez_{ m}u^\ez
 +\ez u^\ez {\rm\, div}(A Du^\ez) +\ez a^{ij}u^\ez_{ i}u^\ez_{ j}\nonumber\\
  &&=2\langle Du^\ez,ADu^\ez\rangle^2+\ez\langle A Du^\ez, Du^\ez\rangle  + u^\ez  \ca_H ^\ez [u^\ez]\nonumber\\
  &&\qquad + 4u^\ez a^{im}   u_m^\ez  u^\eps_{ij} a^{j\ell}u_\ell^\ez
+2u^\ez a_k^{mj} u^\ez_j a^{k\ell}u_\ell^\ez u_m^\ez \nonumber \\
&&=  2\langle Du^\ez,ADu^\ez\rangle^2+\ez \langle A Du^\ez, Du^\ez\rangle\nonumber \\
&&\qquad+  4u^\ez \langle ADu^\ez,D^2u^\ez ADu^\ez\rangle+2u^\ez \langle \langle Du^\ez, DADu^\ez \rangle, ADu^\ez \rangle, \nonumber \label{diff_aron3}
 \end{eqnarray}
where $\langle Du^\ez, DADu^\ez \rangle$ is interpreted as the vector $(\langle Du^\ez, A_kDu^\ez\rangle)_k$ with $A_k$ being the element-wise derivative of $A$.
Choose $\phi\in C^\infty_0(\Omega)$ such that
$$\phi=1 \ {\rm{in}}\ V, \ 0\le \phi \le 1,$$
and, for $\beta>0$ to be determined later, define the auxiliary function $w^\ez$ by
$$w^\ez:=\phi^2 v^\ez+\bz z^\ez.$$
If $w^\ez$ attains its maximum on $\partial\boz$, then
 $$\sup_{\overline V} v^\epsilon\le\sup_{\overline V} w^\ez(x)\le \max_{\overline\Omega} w^\epsilon=\max_{\partial\Omega} w^\epsilon=\frac{\beta}2\max_{\partial\boz} u ^2,$$
hence (\ref{gradient_bound}) holds.
Thus we may assume $w^\ez$ attains its maximum
at an interior point $x_0\in \boz$.
This gives
$$D w^\ez(x_0)=0, D^2w^\ez(x_0)\le 0,$$
so that
\begin{equation}\label{w}
-L_\ez w^\ez(x_0)=  -(2a^{ik}u_k^\ez    a^{j\ell}u_\ell^\ez+\ez a^{ij})w^\ez_{ij}\Big|_{x=x_0}\ge0.
\end{equation}
On the other hand, from (\ref{diff_aron2}) and (\ref{diff_aron3}) we have that, at $x=x_0$,
\begin{eqnarray*}
0 &&\le-L_\ez w^\ez(x^0)= -L_\ez (\phi^2v^\ez)-\beta L_\ez z^\ez\\
&&= -\phi^2L_\ez v^\ez-\beta L_\ez z^\ez- v^\ez L_\ez\phi^2-8\phi a^{ ik}u^\ez_k a^{j\ell}u^\ez_\ell\phi_i \sum_{r=1}^nu^\ez_{rj}u^\ez_r-
4\ez\phi\sum_{m=1}^n\phi_i a^{ij} u^\ez_{mj}u^\ez_m\\
&&=\left[-2 \phi^2|D^2u^\ez ADu^\ez|^2-\ez\phi^2\sum_{s=1}^n a^{ij}u^\ez_{si}u^\ez_{sj}-2\bz\langle Du^\ez,ADu^\ez\rangle^2-\ez\bz \langle Du^\ez,ADu^\ez\rangle\right] \\
&&\quad-\left[4\bz  u^\ez \langle ADu^\ez,D^2u^\ez ADu^\ez\rangle
+2\bz u^\ez a_k^{mj} u^\ez_ju_m^\ez a^{k\ell}u_\ell^\ez \right]  \\
&&\quad-\left[8\phi a^{ ik}u^\ez_k a^{j\ell}u^\ez_\ell\phi_i \sum_{r=1}^nu^\ez_{rj}u^\ez_r+4\ez\phi \sum_{m=1}^n\phi_i a^{ij} u^\ez_{mj}u^\ez_m\right] +\phi^2 
\sum_{s=1}^n u^\ez_sF_{s} - v^\ez L_\ez(\phi^2)\\
&&=I_1+I_2+I_3+I_4+I_5.
\end{eqnarray*}

We estimate $I_1,\cdots, I_5$ as follows.
Since ${\langle \xi,A\xi\rangle}\ge\frac1L|\xi|^2$ {for all $\xi \in \mathbb R^n$}, we have
     \begin{eqnarray*}I_1&&=-2 \phi^2|D^2u^\ez ADu^\ez|^2-\ez\phi^2\sum_{s=1}^n a^{ij}u^\ez_{si}u^\ez_{sj}-2\bz{\langle Du^\ez,ADu^\ez\rangle}^2-\ez\bz \langle A Du^\ez, Du^\ez\rangle\\
     &&\le -2 \phi^2|D^2u^\ez ADu^\ez|^2-\frac{\ez}{L}\phi^2|D^2u^\ez|^2-\frac{2\beta}{L^2}|Du^\ez|^4.
        \end{eqnarray*}
Applying Young's inequality, we can  estimate $I_2$ by
\begin{eqnarray*}
I_2&&= -  4\bz  u^\ez {\langle ADu^\ez,D^2u^\ez ADu^\ez\rangle}
-2\bz u^\ez a_k^{mj} u^\ez_ju_m^\ez a^{k\ell}u_\ell^\ez   \\
&&\le 4\bz |u^\ez| |ADu^\epsilon||D^2u^\ez A Du^\ez|
+ C|Du^\ez|^3\\
&&\le   \bz^{4/3}|D^2u^\ez ADu^\ez|^{4/3} + C |Du^\ez|^4+ C(\bz),
\end{eqnarray*}
 where we have used (\ref{max_est}). {Henceforth} $C>0$ denotes constants depending only on
$n$, {$L$}, $\|A\|_{C^{1,1}(\Omega)}$,
$\displaystyle\|u\|_{C(\overline\Omega)}$, and dist($V,\partial\Omega$).

Similarly,  by Young's inequality we have
 \begin{eqnarray*}
 I_3
 &&=-8\phi a^{ ik}u^\ez_k a^{j\ell}u^\ez_\ell\phi_i {\sum_{r=1}^n}u^\ez_{rj}u^\ez_r- 4\ez\phi {\sum_{m=1}^n}\phi_i a^{ij} u^\ez_{mj}u^\ez_m\\
   &&\le {8}  \phi {\langle AD\phi,Du^\ez\rangle}\cdot {\langle Du^\ez,D^2u^\ez ADu^\ez\rangle} + {4}\ez{\langle AD^2u^\ez Du^\ez, D\phi\rangle}\phi\\
 &&\le C|D^2 u^\ez AD u^\ez ||Du^\ez|^2\phi+ C\ez |D^2 u^\ez D u^\ez |\phi\\
  &&\le \frac18 |D^2 u^\ez ADu^\ez|^2\phi^2+ \frac{\ez}{16 L}|D^2 u^\ez|^2\phi^2+  C|Du^\ez|^4+C.
  \end{eqnarray*}

For $I_4$, {by using $0<\eps \le 1$}, we have
\begin{eqnarray*}
I_4&=& {\sum_{s=1}^n} \left[4\phi^2  u^\ez_sa^{ik}_su_k^\ez  u^\eps_{ij} a^{j\ell}u_\ell^\ez+\phi^2  u^\ez_sa_k^{ij}u^\ez_iu^\ez_j a^{k\ell}_s u_\ell^\ez\right.\\
&&\qquad+\left.\phi^2  u^\ez_s a^{ij}_{sr}u^\ez_iu^\ez_j a^{k\ell}u_\ell^\ez
+\eps\phi^2  u^\ez_s {\,\rm div} (A_s D u^{\ez})\right]\\
&&\le \frac18|D^2u^\ez ADu^\ez|^2\phi^2+ C|Du^\ez|^4+ \frac{\ez}{16L} \phi^2  |D^2u^\ez|^2 +C.
 \end{eqnarray*}
 Finally, for $I_5$, we have
   \begin{eqnarray*}
  &&I_5=
   2  v^\ez a^{ik}u_k^\ez  (\phi^2)_{ij} a^{j\ell}u_\ell^\ez+4  v^\ez a^{ik}(\phi^2)_{k }   u^\eps_{ij} a^{j\ell}u_\ell^\ez
+2  v^\ez a_k^{ij}(\phi^2)_{i }u^\ez_j a^{k\ell}u_\ell^\ez
\\
&&\qquad\qquad+ v^\ez a_k^{ij}u^\ez_iu^\ez_j a^{k\ell}(\phi^2)_{\ell  }
+\ez v^\ez{\rm\,div} (AD\phi^2) \\
&&\le C|Du^\ez|^4+C |D^2u^\ez ADu^\ez||Du^\ez|^2\phi+ C\ez|Du^\ez|^2\\
&&\le \frac18  |D^2u^\ez A Du^\ez|^2\phi^2+C|Du^\ez|^4+C.
    \end{eqnarray*}
Combining all these estimates with~\eqref{w} yields that, at $x=x_0$,
       \begin{eqnarray*}
    &&2 \phi^2|D^2u^\ez ADu^\ez|^2+\frac{\ez}L\phi^2|D^2u^\ez|^2+\frac2{L^2}\bz|Du^\ez|^4\\
&&\le |D^2u^\ez ADu^\ez|^2\phi^2+ C|Du^\ez|^4+ C\bz^{4/3} |D^2u^\ez ADu^\ez|^{4/3}+\frac{\ez}{8 L}
\phi^2|D^2 u^\ez|^2+C(\bz),
        \end{eqnarray*}
so that
\begin{eqnarray*}
  |D^2u^\ez ADu^\ez|^2\phi^2+ \frac2{L^2}\bz|Du^\ez|^4\le  C|Du^\ez|^4+ C\bz^{4/3}|D^2u^\ez ADu^\ez|^{4/3}+C(\bz).
 \end{eqnarray*}
We may choose $\beta>1$ sufficiently large so that
  \begin{eqnarray*}
|D^2u^\ez ADu^\ez|^2\phi^2 +   \frac{\bz}
{L^2}| Du^\ez|^4\le   C\bz^{4/3} |D^2u^\ez ADu^\ez|^{4/3}+C(\bz).
 \end{eqnarray*}
Multiplying both sides of this inequality by $\phi^4$ and applying Young's inequality
implies
\begin{eqnarray*}
|D^2u^\ez ADu^\ez|^2\phi^6 + \frac{\bz}{L^2}| Du^\ez|^4\phi^4&&\le   C\bz^{4/3} |D^2u^\ez ADu^\ez|^{4/3}\phi^4+C(\bz)\\
&&\le \frac12 |D^2u^\ez ADu^\ez|^2\phi^6+C(\bz).
\end{eqnarray*}
Hence we have
         \begin{eqnarray*}
      | Du^\ez|^4\phi^4\Big|_{x=x_0}\le    C.
        \end{eqnarray*}
This finishes the proof{, since $v^\ez=\frac12 |Du^\ez|^2$ attains its maximum at $x^0$}.
\end{proof}

Next we will establish the boundary H\"older continuity estimate of $u^\epsilon$.

\begin{thm}\label{l2.3} With the same notations of Theorem \ref{first_lemma},
assume that in addition $L<2^{1/4}$.
Then there exist $\delta_0>0$, $\epsilon_0>0$, 
$\gamma\in (0,1)$, and $C>0$ depending only on $\Omega$ and $\displaystyle \|A\|_{L^\infty(\Omega)}$
such that if $\| DA\|_{L^\infty(\boz)}\le \delta_0$ and $0<\ez<\epsilon_0$,
then 
\begin{equation}\label{e2.2}|u^\ez(x)-u(y_0)|\le C |x-y_0|^{\gamma}, \ y_0 \in\partial\boz, \ x\in\boz.
\end{equation}

\end{thm}

\begin{proof} 
To show \eqref{e2.2}, assume for simplicity that $y_0=0\in\partial\Omega$. Define
$w(x)=\lz |x|^\gamma${, where $\lz>1$ is chosen such that}
$$-w+u(0)\le u\le u(0)+w \ {\rm{on}}\ \partial \boz.$$ {This is always possible, since $u$ is Lipschitz.}
Now we claim that $w$ is a supersolution of the {$\ez$-regularized} equation (\ref{e2.x2}).
In fact, direct calculations imply
\begin{eqnarray*}
-a^{ik}(x)w_k(x)w_{ij}(x)a^{j\ell}(x)w_\ell(x)&&=-\frac{\lz^2\gamma^2 a^{ik}x_ka^{j\ell}x_\ell}{|x|^{4-2\gamma}}\cdot\lz\gamma\lf[(\gamma-2)\frac{x_ix_j}{|x|^{4-\gamma}}+\frac{\dz_{ij}}{|x|^{2-\gamma}}\r]\\
&&=\lz^3\gamma^3(2-\gamma)\frac{ \langle x, Ax\rangle^2 }{|x|^{8-3\gamma}}
- \lz^3\gamma^3\frac{ \langle x, A^2x\rangle}{|x|^{6-3\gamma}}\\
&&\ge  \lz^3\gamma^3\frac{2-\gamma}{L^2}|x|^{3\gamma-4}  - \lz^3\gamma^3L^2|x|^{3\gamma-4}\\
&&=   \lz^3\gamma^3\left(\frac{2-\gamma}{L^2}-L^2\right)|x|^{3\gamma-4}.
                \end{eqnarray*}
Note that we can choose $\gamma>0$ so that $\wz\gamma:=\frac{2-\gamma}{L^2}-L^2>0$,
since $L<2^{\frac14}$.
                Next we estimate
                \begin{eqnarray*}
                -a_k^{ij}(x)w_i(x)w_j(x) a^{k\ell}(x)w_\ell(x)&&=-\lz^3\gamma^3a_k^{ij}(x)a^{k\ell}(x)\frac{x_ix_jx_\ell}{|x|^{6-3\gamma}} \\
                &&\ge -\lz^3\gamma^3\|A\|_{L^\infty(\overline \Omega)}\|DA\|_{L^\infty(\overline \Omega)}
|x|^{3\gamma-3}
                \end{eqnarray*}
                Finally,  {for the regularization term} we can estimate
                   \begin{eqnarray*}-\ez  {\rm\, div}(A Dw)(x)&&=- \ez\lz  a^{ij}\gamma\frac{\dz_{ij}}{|x|^{{2-\gamma}}}- \ez\lz  a^{ij}\gamma(\gamma-2)\frac{x_ix_j }
                   {|x|^{{4-\gamma}}}
                   - \ez\lz\gamma a_{j}^{ij}\frac {x_i}{|x|^{2- \gamma}} \\
                   &&\ge - \ez \lz L\gamma(n+ {\gamma-2})|x|^{\gamma-2}-2\ez\lz n \gamma\|DA\|_{L^\fz(\overline\boz)}|x|^{\gamma-1}.
                         \end{eqnarray*}
 Putting these estimates together, we have
 \begin{eqnarray*}
&&-\mathcal A_H^\epsilon[w]\\
&&\ge2\lz^3\gamma^3\wz\gamma|x|^{3\gamma-4}-\lz^3\gamma^3\|A\|_{L^\infty(\Omega)}\| DA\|_{L^\infty(\Omega)}|x|^{3\gamma-3}
-2\ez \lz L\gamma(n+\gamma-2)|x|^{\gamma-2}\\
&&\qquad-2\ez\lz n \gamma\|DA\|_{L^\fz(\boz)}|x|^{\gamma-1}\\
&&\ge2\lz^3\gamma^3\wz\gamma|x|^{3\gamma-4}-\lz^3\gamma^3\|A\|_{L^\infty(\Omega)}\|DA\|_{L^\infty(\Omega)}|x|^{3\gamma-3}
-C\epsilon|x|^{3\gamma-4}.
                \end{eqnarray*}
Set
\[
\delta_0:=\dz(\boz, A)=\frac{\min_{x\in\overline\boz}\frac{\wz \gamma}{2|x| }}{\|A\|_{L^\infty(\overline \Omega)}}.
\]
If $\displaystyle \|  DA\|_{L^\infty(\boz)}\le \dz_0$ and $\epsilon_0>0$ is sufficiently small, then
we  have $\gamma \in (0,1)$ that
$$ -\ca_H^\epsilon [w]\ge0.$$
By the comparison principle, we conclude that $w+u(0)\ge u^\ez$ in $\boz$.
Similarly, we have  $-w+u(0)\ge u^\ez$ in $\boz$.
Thus we obtain
$$|u^\ez(x)-u(0)|\le \lz |x|^{\gz},\ \ x \in \Omega. $$
This completes the proof.
\end{proof}

\subsection{Flatness estimates}

In this section, we will prove refined { a priori} estimates of the $\epsilon$-{regularized} equation (\ref{e2.x2}) {under a flatness assumption}.
Assume $u^\ez\in C^\infty(\Omega)\cap C(\overline\Omega)$ is a smooth solution to the $\epsilon$-{regularized} equation~\eqref{e2.x2} associated with  
$A\in\mathscr A(\Omega)\cap C^{\infty}(\Omega)$.

\begin{thm}\label{l2.4} Assume $B(0,3)\subset\Omega$. For any $0<\lambda<1$,
if $A\in\mathscr A(\Omega)\cap C^{\infty}(\Omega)$ satisfies
$A(0)=I_n$ and
 \begin{equation}\label{lambda_assumption}
 \| DA\|_{L^\infty(B(0,\,3))}+\| D^2A\|_{{L^{\infty}(B(0,\,3))}}\le\lz,
 \end{equation}
and if $u^\epsilon\in C^{\infty}(\Omega)$ is a smooth solution of (\ref{e2.x2}) that satisfies
\begin{equation}\label{flatness_assumption}\max_{x\in B(0,2)}|u^\ez(x)-x_n|\le\lz,
\end{equation}
then there exists a constant $C>0$ independent of $\ez $ and $\lz$ such that
\begin{equation}
|Du^\ez(x)|^2\le u^\ez_{n}(x)+C\lz^{1/2} \quad {\text{for all}}\ x\in B(0,1).\label{flat}
\end{equation}
\end{thm}

\begin{proof}
Set   $\Phi(p):=(|p|^2-p_n)^2_+={\max\{ |p|^2-p_n , 0\}^2}$. 
Let $\phi \in C^\infty_0(B(0,3))$  {be} such that
$$\phi=1\ {\rm{in}}\ B(0,1), \ \phi=0 \ {\rm{outside}}\ B(0,2),\
0\le\phi\le1, \ {\rm{and}}\ |D\phi| \le 2.$$
Define $$v^\ez=\phi^2\Phi(Du^\ez)+\bz(u^\ez-x_n)^2+\lz|Du^\ez|^2.$$
Applying Theorem 3.1, we have
 $$|u^\ez|+|Du^\ez|\le C\ {\rm{in}}\ B(0,2).
$$
 If $\displaystyle\max_{B(0,\,2)}v^\ez$ is attained on $\partial B(0, 2)$, then by (\ref{max_est}),
 (\ref{lambda_assumption}), and (\ref{flatness_assumption}) we have
 $$\max_{B(0,\,2)}v^\ez(x)=\max_{\partial B(0,2)}\big(\beta (u^\epsilon-x_n)^2+\lambda |Du^\epsilon|^2\big)\le \bz\lz^2+C\lz\le C\lz,$$
 {and} hence
 $$\max_{B(0,1)}\big(|Du^\epsilon|^2-u_{x_n}^\epsilon\big){_+}^2
 \le
\max_{B(0,\,1)}\Phi(Du^\ez )\le  C\lz$$
so that (\ref{flat}) holds.
 {Therefore} we may assume  that $v^\ez$ attains its  maximum
at an interior point $x_0\in B(0,2)$.
If $\big(|Du^\ez|^2-u^\ez_n\big)(x_0)\le0$, then $\Phi(Du^\ez)(x_0)=0$ and
  $$\max_{B(0,1)}\Phi(Du^\ez )\le\max_{B(0,1)}v^\ez(x)=v^\ez(x_0)
\le v^\ez(x^0)\le \bz\lz^2+C\lz\le C\lz
  $$
so that (\ref{flat}) also holds.
So we  {can also} assume
$$\big(|Du^\ez|^2-u^\ez_n\big)(x_0)>0.$$

To estimate $v^\ez(x_0)$, let $L_\ez$ and $F_s^\epsilon$ be given by (\ref{L-ep}) and (\ref{F-ep}). {We need to compute $L_\ez v^\ez$ at $x^0$. Using
$$\ca_H [u^\ez]+\ez{\rm div}(A D u^\ez)=2a^{ik}u_k^\ez   u^\ez_{ij} a^{j\ell}u_\ell^\ez+a_k^{ij}u^\ez_iu^\ez_j a^{k\ell}u^\ez_\ell+\ez{\rm div}(A D u^\ez)=    0,$$
we obtain
\begin{eqnarray*}-L_\ez  ((u^\ez-x_n)^2)
&&= -4a^{ik}u_k^\ez   u^\ez_{ij} a^{j\ell}u_\ell^\ez(u^\ez-x_n)
 - 4a^{ik}u_k^\ez   a^{j\ell}u_\ell^\ez(u^\ez_{i }-\dz_{in})(u^\ez_{j }-\dz_{jn})\\
 &&\quad-8a^{ik}  (u^\ez_{k }-\dz_{kn})  u^\eps_{ij} a^{j\ell}u_\ell^\ez(u^\ez-x_n)\\
 &&\quad-4a_k^{ij}(u^\ez_{i }-\dz_{in}) u^\ez_j a^{k\ell}u_\ell^\ez(u^\ez-x_n)\\
 &&\quad-2a_k^{ij}u^\ez_iu^\ez_j a^{k\ell} (u^\ez_{\ell }-\dz_{\ell n})(u^\ez-x_n)\\
 &&\quad-2\eps(u^\ez-x_n) {\,\rm div} (AD u^\ez-ADx_n)-2\eps{\langle D u^\ez -e_n,A(D u^\ez -e_n)\rangle}\\
&&=   - 4 \big({\langle D u^\ez,AD u^\ez\rangle}-a^{nk}u^\ez_k\big)^2- 2\eps{\langle D u^\ez -e_n,A(D u^\ez -e_n)\rangle}  \\
&&\quad -8a^{ik}  (u^\ez_{k }-\dz_{kn})  u^\eps_{ij} a^{j\ell}u_\ell^\ez(u^\ez-x_n)\\
 &&\quad-4a_k^{ij}(u^\ez_{i }-\dz_{in}) u^\ez_j a^{k\ell}u_\ell^\ez(u^\ez-x_n)\\
 &&\quad+2a_k^{ij}u^\ez_iu^\ez_j a^{k\ell}\dz_{\ell n}  (u^\ez-x_n)+2\epsilon {\sum_{i=1}^n}a^{in}_i(u^\ez-x_n) \\
&&=J_1+J_2+J_3+J_4+J_5+J_6,
\end{eqnarray*}
where we denote
$e_n=(0, ..., 0, 1)$.

Applying \eqref{flatness_assumption} and Theorem 3.1, we have {by straightforward calculations} that
\begin{align*}
|J_3|
&\le C\lambda |D^2u^\ez ADu^\ez|,
\end{align*}
and
\begin{align*}
|J_4|, |J_5|&\le  C\lambda,
\end{align*}
as well as
$$
|J_6|\le C\epsilon\lambda.
$$
Since $\|DA\|_{L^\infty}\le \lz$ and $A(0)=I_n$, we have
$|A-I_n|\le C\lambda$ on $\Omega$ and hence
\begin{eqnarray*}
\big|{\langle Du^\ez, ADu^\ez\rangle}-a^{nk}u^\ez_k\big|&\ge& \big||Du^\ez|^2-u^\ez_n\big|-\big|{\langle Du^\ez, (A-I_n)Du^\ez\rangle}\big|\\
&&-\big|a^{nn}-1\big||u_n| -\sum_{k=1}^{n-1}\big|a^{nk}u^\ez_k\big|\\
&&\ge \big||Du^\ez|^2-u^\ez_n\big|-C\lz.
\end{eqnarray*}
Hence we have {that} 
$$J_1=-4\big({\langle Du^\ez, ADu^\ez\rangle}-a^{nk}u_k^\epsilon\big)^2
\le -4\big||Du^\ez|^2-u^\ez_n\big|^2+C\lz.
$$
Since ${\langle \xi, A\xi\rangle}\ge \frac1{L}|\xi|^2$, we {also} have
$$
J_2\le -\frac{\epsilon}{L}\big|Du^\epsilon-e_n\big|^2.
$$
{Combining} all these estimates on $J_i$'s, we have
\begin{eqnarray}\label{flat0} -L_\ez  \big((u^\ez-x_n)^2\big)
&&\le - 4\big(|Du^\ez|^2-u^\ez_n\big)^2  {\color{magenta}-}\frac{2\ez}{L}  |D u^\ez -e_n|^2\nonumber\\
&&\quad+C\lz(1+|D^2u^\ez ADu^\ez|).
\end{eqnarray}
Moreover, similar to the proof of Theorem 3.1, we have
\begin{eqnarray} &&\frac12L_\ez \big(|Du^\ez|^2\big)
=  2 |D^2u^\ez ADu^\ez|^2+ \eps  {\sum_{s=1}^n}\big(a^{ij}u^\ez_{si}u^\ez_{sj}-  u^\ez_s F_s^\ez
\big)
\nonumber\\
&&\ge 2 |D^2u^\ez ADu^\ez|^2+ \frac{\ez}{L}|D^2u^\ez|^2- C|D^2u^\ez ADu^\ez||Du^\ez|^2-C |Du^\ez|^4\nonumber\\
&&\ge  |D^2u^\ez ADu^\ez|^2+ \frac{\ez}{L}|D^2u^\ez|^2-C.
\label{flat1}
\end{eqnarray}
Next we need to estimate $L_\ez(\phi^2 \Phi(Du^\ez))$. First recall
\begin{eqnarray*}
L_\ez(\Phi(Du^\ez))&&=  2a^{ik}u_k^\ez    a^{j\ell}u_\ell^\ez (\Phi(Du^\ez))_{ij }+\ez  {\rm\, div} (A D(\Phi(Du^\ez)))\\
&&+\big(4a^{is}    u^\eps_{ij} a^{j\ell}u_\ell^\ez
+2a_k^{sj} u^\ez_j a^{k\ell}u_\ell^\ez+a_k^{ij}u^\ez_iu^\ez_j a^{ks} \big)(\Phi(Du^\ez))_{s }.
\end{eqnarray*}
{As explained earlier, we may assume} $|Du^\ez|^2>u_n^\ez$ at $x^0\in B(0,2)$. {With this assumption we have at $x=x^0$ that}
$$(\Phi(Du^\ez))_{s }= 2\big(|Du^\ez|^2-u^\ez_n\big)\Big( 2{\sum_{k=1}^n}u^\ez_{ks}u^\ez_k-u^\ez_{ns}\Big),$$
and
\begin{eqnarray*} (\Phi(Du^\ez))_{ij }&&= 2\Big( 2{\sum_{s=1}^n}u^\ez_{sj}u^\ez_s-u^\ez_{nj}\Big)\Big( 2{\sum_{s=1}^n}u^\ez_{si}u^\ez_s-u^\ez_{ni}\Big)\\
&&\quad+
2\big(|Du^\ez|^2-u^\ez_n\big)\Big(2{\sum_{s=1}^n}(u^\ez_{si}u^\ez_{sj}
+u^\ez_{sij}u^\ez_s)-u^\ez_{nij}\Big).
\end{eqnarray*}
Hence we obtain that, at $x=x_0$,
 \begin{equation}\label{phi_du}
 \begin{split}
 L_\ez(\Phi(Du^\ez))&= 4a^{ik}u_k^\ez    a^{j\ell}u_\ell^\ez\Big( 2{\sum_{s=1}^n}u^\ez_{s j}u^\ez_s-u^\ez_{nj}\Big)\Big(2{\sum_{s=1}^n}u^\ez_{si}u^\ez_s-u^\ez_{ni}\Big)\\
&\quad+
4\big(|Du^\ez|^2-u^\ez_n\big) a^{ik}u_k^\ez    a^{j\ell}u_\ell^\ez
\Big( 2{\sum_{s=1}^n}(u^\ez_{si}u^\ez_{sj}+u^\ez_{sij}u^\ez_s)-u^\ez_{nij}\Big)\\
&\quad+2\ez a^{ij}  \Big( 2{\sum_{s=1}^n}u^\ez_{si}u^\ez_s-u^\ez_{ni}\Big)
\Big( 2{\sum_{s=1}^n}u^\ez_{sj}u^\ez_s-u^\ez_{nj}\Big)\\
&\quad+
2\ez\big(|Du^\ez|^2-u^\ez_n\big)a^{ij}
\Big( 2{\sum_{s=1}^n}(u^\ez_{si}u^\ez_{sj}+u^\ez_{sij}u^\ez_s)-u^\ez_{nij}\Big)\\
& \quad+ 2\ez a^{ij}_j\big(|Du^\ez|^2-u^\ez_n\big)\Big( 2{\sum_{s=1}^n}u^\ez_{sj}u^\ez_s-u^\ez_{nj}\Big) \\
&\quad +2\big(|Du^\ez|^2-u^\ez_n\big){\sum_{m=1}^n}G_m^\epsilon \Big(2{\sum_{s=1}^n}u^\ez_{sm}u^\ez_s-u^\ez_{nm}\Big)\\
&=4a^{ik}u_k^\ez    a^{j\ell}u_\ell^\ez\Big(2{\sum_{s=1}^n}u^\ez_{s j}u^\ez_s-u^\ez_{nj}\Big)\Big( 2{\sum_{s=1}^n}u^\ez_{si}u^\ez_s-u^\ez_{ni}\Big)\\
&\quad+8\big(|Du^\ez|^2-u^\ez_n\big) a^{ik}u_k^\ez    a^{j\ell}u_\ell^\ez
\Big( {\sum_{s=1}^n}u^\ez_{si}u^\ez_{sj}\Big)\\
&\quad+2\ez a^{ij}  \Big( 2{\sum_{s=1}^n}u^\ez_{si}u^\ez_s-u^\ez_{ni}\Big)
\Big( 2{\sum_{s=1}^n}u^\ez_{sj}u^\ez_s-u^\ez_{nj}\Big)\\
&\quad+4\ez a^{ij}  \Big(|Du^\epsilon|^2-u^\ez_{n}\Big)
\Big({\sum_{s=1}^n} u^\ez_{sj}u^\ez_{sj}\Big)\\
&\quad+2\big(|Du^\epsilon|^2-u_n^\epsilon\big)\Big( 2{\sum_{s=1}^n}u_s^\epsilon L_\epsilon(u^\epsilon_s)-L_\epsilon(u^\epsilon_n)\Big)\\
&=K_1+K_2+K_3+K_4+K_5.
\end{split}
\end{equation}
Here $G_m^\ez$ is as defined in~\eqref{G-ep}. Now we estimate $K_1, ..., K_5$ separately as follows.
For $K_1$, we have

\begin{align*}
K_1
&=4\Big[2\langle Du^\ez, D^2u^\ez ADu^\ez\rangle-\langle (D^2u^\ez)^n,ADu^\ez\rangle\Big]^2,
\end{align*}
where $(D^2u^\ez)^n$ denotes the $n^{th}$-row of $D^2u^\ez$.
For $K_2$, we have
\[
K_2=8(|Du^\ez|^2-u^\ez_n)|D^2u^\ez ADu^\ez|^2.
\]
For $K_3$, we have
$$
K_3
\ge \frac{2\ez}L{\sum_{i=1}^n} \Big(2{\sum_{s=1}^n}u^\ez_{si}u^\ez_s-u^\ez_{ni}\Big)^2.$$
For $K_4$, we have
\[
 K_4\ge\frac{4\ez}{L} \big(|Du^\ez|^2-u^\ez_n\big) \big|D^2u^\ez\big|^2.
\]
From (\ref{diff_aron1}), we have
$$
K_5=2\big(|Du^\epsilon|^2-u_n^\epsilon\big)\Big(\sum_{s=1}^n 2u_s^\epsilon F_s^\epsilon-F_n^\epsilon\Big),
$$
so that we can apply Theorem 3.1 to estimate
$$
\big|K_5\big|
\le \big(|Du^\epsilon|^2-u_n^\epsilon\big)
\big(C\lz   |D^2u^\ez ADu^\ez|+ \frac{\ez}{4L}|D^2u^\ez|^2 +C\lz\big).
$$
Putting these estimates into~\eqref{phi_du} gives
 \begin{align}\label{le-est1}
 L_\ez(\Phi(Du^\ez)) \ge&\   8\big(|Du^\ez|^2-u^\ez_n\big)
\Big(|D^2 u^\ez AD u^\ez|^2 + \frac\ez{4L} |D^2 u^\ez   |^2\Big)\\
&+ 4\Big[2{\langle Du^\ez, D^2u^\ez ADu^\ez\rangle-\langle (D^2u^\ez)^n,ADu^\ez\rangle}\Big]^2\nonumber\\
&+\frac{2\ez}{L} \sum_{i=1}^n \Big( 2{\sum_{s=1}^n}u^\ez_{si}u^\ez_s-u^\ez_{ni}\Big)^2\nonumber\\
&- C\lz (|Du^\ez|^2-u^\ez_n)  |D^2u^\ez ADu^\ez| -  C\lz.\nonumber
\end{align}
It follows from (\ref{le-est1}) that
 \begin{eqnarray*}
L_\ez\big(\phi^2\Phi(Du^\ez)\big)&=& \phi^2 L_\ez\big(\Phi(Du^\ez)\big)+\Phi(Du^\ez) L_\ez\big(\phi^2\big) \\
&&+4a^{ik}u_k^\ez a^{jl}u_l^\ez \phi\phi_i(\Phi(Du^\ez))_j+2\ez\phi a^{ij}\phi_i(\Phi(Du^\ez))_j \\
&\ge& 8\phi^2\big(|Du^\ez|^2-u^\ez_n\big)\big|D^2 u^\ez AD u^\ez\big|^2
+\Phi\big(Du^\ez\big) L_\ez\big(\phi^2\big)\\
&&+4\phi^2\Big[2{\langle Du^\ez, D^2u^\ez ADu^\ez\rangle-\langle (D^2u^\ez)^n,ADu^\ez\rangle}\Big]^2\\
&&+4a^{ik}u^\ez_ka^{j\ell}u^\ez_\ell\phi\phi_i(\Phi(Du^\ez))_j
+\frac{2\ez}L \phi^2\sum_{i=1}^n \Big(2{\sum_{s=1}^n}u^\ez_{si}u^\ez_s-u^\ez_{ni}\Big)^2\\
&&+2\ez\phi a^{ij}\phi_i(\Phi(Du^\ez))_j - C\lz \phi^2\Big[1+\big(|Du^\ez|^2-u^\ez_n\big)
\big|D^2u^\ez ADu^\ez\big| \Big].
 \end{eqnarray*}
 It is easy to see that
 \begin{eqnarray*}
{|L_\ez \big(\phi^2\big)|}&&=  {\Big|}2a^{ik}u_k^\ez    a^{j\ell}u_\ell^\ez (\phi^2)_{ij }
+\ez {\rm div}(A D  \phi^2)\\
&&\quad+\Big(4a^{is}    u^\eps_{ij} a^{j\ell}u_\ell^\ez
+2a_k^{sj} u^\ez_j a^{k\ell}u_\ell^\ez+a_k^{ij}u^\ez_iu^\ez_j a^{ks}\Big)\big(\phi^2\big)_{s}{\Big|}\\ 
&&\le C|Du^\ez|^2+ \phi\big|D^2u^\ez A Du^\ez\big|+C\ez\\
&&\le \phi\big|D^2u^\ez A Du^\ez\big|+C,
\end{eqnarray*}
so that
$$\Phi\big(Du^\ez\big){|L_\ez \big(\phi^2\big)|}\le \big(|Du^\ez|^2-u_n^\ez\big)^2
\big(\phi|D^2u^\ez ADu^\ez|+C\big).$$
By Young's inequality, we have
\begin{align*}
&4a^{ik}u^\epsilon_k a^{jl}u_l^\epsilon \phi\phi_i (\Phi(Du^\epsilon))_j
\\
&=8a^{ik}u^\ez_ka^{j\ell}u^\ez_\ell\phi\phi_i\big(|Du^\ez|^2-u^\ez_n\big)\Big(  2{\sum_{s=1}^n}u^\ez_{sj}u^\ez_s-u^\ez_{nj}\Big) \\
&=8a^{ik}u^\ez_k\phi\phi_i(|Du^\ez|^2-u^\ez_n)\cdot
  \Big(2{\langle Du^\ez, D^2u^\ez ADu^\ez\rangle-\langle (D^2u^\ez)^n,ADu^\ez\rangle}\Big) \\
   &\le 4\phi^2\Big[2{\langle Du^\ez, D^2u^\ez ADu^\ez\rangle-\langle (D^2u^\ez)^n,ADu^\ez\rangle}\Big]^2\\
&\quad+ 16\Big[{\langle D\phi, ADu^\ez\rangle}(|Du^\ez|^2-u^\ez_n)\Big]^2.
\end{align*}
Thus by Theorem 3.1, we obtain
   \begin{eqnarray*}
  &&4\phi^2\Big[2{\langle Du^\ez, D^2u^\ez ADu^\ez\rangle-\langle (D^2u^\ez)^n,ADu^\ez\rangle}\Big]^2 +4a^{ik}u^\ez_ka^{j\ell}u^\ez_\ell\phi\phi_i(\Phi(Du^\ez))_j\\
   &&\ge - 16\Big[{\langle D\phi, ADu^\ez\rangle} (|Du^\ez|^2-u^\ez_n)\Big]^2\\
   &&\ge -  C \big(|Du^\ez|^2-u^\ez_n\big)^2.
    \end{eqnarray*}
Similarly, by Young's inequality,  we have that
\begin{align*}
&2\ez\phi a^{ij}\phi_i\big(\Phi(Du^\ez)\big)_j = 4\ez\phi a^{ij}\phi_i\big(|Du^\ez|^2-u^\ez_n\big)\Big( 2{\sum_{s=1}^n}u^\ez_{sj}u^\ez_s-u^\ez_{nj}\Big) \\
&\le C\ez|D\phi|^2\big(|Du^\ez|^2-u^\ez_n\big)^2+ \frac{\ez}{L}\phi^2\sum_{i=1}^n\Big( 2{\sum_{s=1}^n}u^\ez_{si}u^\ez_s-u^\ez_{ni}\Big)^2,
\end{align*}
which gives
       \begin{eqnarray*}
&& \frac{2\ez}{L}\sum_{i=1}^n \Big( 2{\sum_{s=1}^n}u^\ez_{si}u^\ez_s-u^\ez_{ni}\Big){^2}\phi^2-2\ez\phi a^{ij}\phi_i\big(\Phi(Du^\ez)\big)_j \\
&&\ge  - C\ez|D\phi|^2\big(|Du^\ez|^2-u^\ez_n\big)^2\\
&&\ge -C\ez\big(|Du^\ez|^2-u^\ez_n\big)^2.
 \end{eqnarray*}
Putting all these estimates together and applying Young's inequality, we conclude that
  \begin{eqnarray}
 L_\ez\big(\phi^2\Phi(Du^\ez)\big)
&&\ge  8\phi^2\big(|Du^\ez|^2-u^\ez_n\big)\big|D^2 u^\ez AD u^\ez\big|^2
- C\big(|Du^\ez|^2-u^\ez_n\big)^2 \nonumber \\
 &&\quad- \big(|Du^\ez|^2-u_n^\ez\big)^2\big(\phi|D^2u^\ez ADu^\ez|+C\big)\nonumber\\
 &&\quad- C\lz  \big(|Du^\ez|^2-u^\ez_n\big)\big|D^2u^\ez ADu^\ez\big|\phi^2
-C\lz\phi^2\nonumber\\
 &&\ge -C(|Du^\ez|^2-u^\ez_n) ^3-C(|Du^\ez|^2-u^\ez_n) ^2-C\lz(|Du^\ez|^2-u^\ez_n)-C\lz\phi^2\nonumber\\
 &&\ge  -C(|Du^\ez|^2-u^\ez_n) ^2-C\lz(|Du^\ez|^2-u^\ez_n)-C\lz. \label{le-est2}
 \end{eqnarray}
Combining the estimates (\ref{flat0}), (\ref{flat1}), with (\ref{le-est2})
yields that, at $x=x_0$,
   \begin{eqnarray*}
  0&&\le-L_\ez\big(v^\ez\big)
= -L_\ez\big(\phi^2 \Phi(Du^\ez)\big)-\bz L_\ez\big((u^\ez-x_n)^2\big)
-\lz L_\ez\big(|Du^\ez|^2\big)\\
 &&\le C\big(|Du^\ez|^2-u^\ez_n\big) ^2+C\lz\big(|Du^\ez|^2-u^\ez_n\big)+C\lz\\
 &&\quad -4\bz\big(|Du^\ez|^2- u^\ez_n\big)^2-{\frac{2\ez\bz}L}  \big|D u^\ez -e_n\big|^2+C\bz\lz+C\bz\lz\big|D^2u^\ez ADu^\ez\big|\\
  &&\quad+2\lz \Big(-|D^2u^\ez ADu^\ez|^2- \frac{\ez}{{L^2}} |D^2u^\ez|^2+C\Big).
 \end{eqnarray*}
 Thus we have that, at $x=x_0$,
    \begin{eqnarray*}
    &&(4\bz-C)\big(|Du^\ez|^2-u^\ez_n\big)^2+  2\lz \big|D^2u^\ez ADu^\ez\big|^2+\frac{2\lz\ez}{{L^2}} \big|D^2u^\ez\big|^2\\
    &&\quad\le   C\lz\big(|Du^\ez|^2-u^\ez_n\big)
 + C(1+\bz) \lz+ C\bz\lz\big|D^2u^\ez ADu^\ez\big|.
 \end{eqnarray*}
Choosing $\bz> C $ and applying Young's inequality, we obtain
     \begin{eqnarray*}
    \bz \big(|Du^\ez|^2-u^\ez_n\big)^2\le  C\lz+2\bz^2\lz .
 \end{eqnarray*}
 Thus we conclude that, at $x=x_0$,
      \begin{eqnarray*}
     \big(|Du^\ez|^2-u^\ez_n\big)^2  \le  C\lz.
 \end{eqnarray*}
 This completes the proof.}
\end{proof}

\section{Differentiability}\label{s5}
This section is devoted to the proof of Theorem 1.1. In order to do it, we need some  lemmas.
The first lemma is the linear approximation property (see also \cite{ksz} Theorem 5.1).

\begin{lem}\label{t3.1}
Let $A\in\mathscr A(\Omega)\cap C(\overline\Omega)$ and $u\in C^{0,1}(\Omega)$ be an absolute minimizer of $\mathscr F_\infty$ with respect to $A$ in $\boz$.
Then for each $x\in\boz$ and every sequence $\{r_j\}_{j\in\nn}$ converging to $0$,
there exists a subsequence ${\bf r}=\{r_{j_k}\}_{k\in\nn}$ and a vector ${\bf e}_{x,\bf r}\in\mathbb R^n$ such that
\begin{equation}\label{e3.1}
\lim_{k\to\fz}\max_{y\in B(0,\,1)}\lf|
\frac{u(x+r_{j_k}y)-u(x)}{r_{j_k}}-\langle {\bf e}_{x,\,\bf r},\, y\rangle\r|=0,
\end{equation}
and $H(x,\,{\bf e}_{x,\,\bf r})= \lip_{d_A}  u(x)$. Here
$$\lip_{d_A}u(x):=\limsup_{y\to x}\frac{|u(x)-u(y)|}{d_A(x,\,y)},$$
and
$$d_A(x,y):=\sup\Big\{w(x)-w(y):\ w\in C^{0,1}(\Omega)
\ {\rm{satisfies}}\ H(z, Dw(z))\le1 \ {\rm{a.e.}}\ z\in\boz\Big\}.$$
\end{lem}

\begin{proof}[Sketch of the proof of Lemma \ref{t3.1}]

Without loss of generality,  assume $x=0\in \Omega$ and  $u(0)=0$.
We also assume  ${ \lip_{d_A}} u(0)>0$, since the case
${ \lip_{d_A}} u(0)=0$ is trivial.

For any fixed $r_0\in\big(0, d_A(0,\partial \boz)\big)$,
assume that $r_{j+1}<r_j<r_0$ for all $j$.
For each $j\in\nn$, define
$$u_j(y)=\frac{1}{r_j}u(r_j y), \ A_j(y)=A(r_j y), \  y\in B\big(0, r_j^{-1}r_0\big),$$
$$A_\infty(y)=A(0), \  y\in\mathbb R^n,$$
and
$$H_j(x, \xi)=\langle A_j(x)\xi,\,\xi\rangle,\ x\in B\big(0,r_j^{-1}r_0\big), \ \xi\in\mathbb R^n.$$
Also let $d_j$ denote the intrinsic distance $d_{A_j}$ corresponding to { $A_j $}.

Recall that by \cite{ksz} Lemma 5.1
there exists $u_\fz\in W^{1,\fz}(\rn)$ and
 a subsequence  $\{r_{j_k}\}_{k\in\nn}$ of $\{r_j\}_{j\in\nn}$
 such that  $ u_{j_k}$  converges to
 $u_\fz $ locally uniformly  in $\mathbb R^n$, and  weak$^*$ in $W^{1,\,\fz}(\rn)$.
Moreover, by \cite{ksz} Lemma 5.5 that
there exists  a vector ${\bf e}\in\rn$ such that
$$u_\fz(x)=\langle {\bf e}, x\rangle, \ x\in\rn,
\ {\rm{and}}\ H_\fz({\bf e})\big(\equiv H(0, {\bf e})\big)=\lip_{d_\fz}u_\fz(0).
$$
From this, we conclude that
$$ \sup_{y\in B(0,1)}\big|\frac{1}{r_{j_k}}u (r_{j_k}y)-\langle {\bf e},\,y\rangle\big|=
\sup_{y\in B(0,1)}\big|u_{j_k}(y)-\langle {\bf e},y\rangle\big|
= \sup_{y\in B(0,1)}\big|u_{j_k}(y)-u_\fz(y)\big|\to 0$$
as $k\to\fz$, and $H_\fz({\bf e})=\lip_{d_A}u(0)$.
This completes the proof.
\end{proof}

Given a pair of functions $A\in \mathscr A(\Omega)\cap C(\overline\Omega)$ and
$u\in C^{0,1}(\Omega)$, and a pair of $0\not=r\in\mathbb R$ and $x_0\in \Omega$,
we define
$$A_{x_0, r}(y)=A(x_0+ry),
\  u_{x_0,r}(y)=\frac{u(x _0+r  y)-u(x_0)}{r }\ \ y\in \Omega_{x_0, r}:=r^{-1}\big(\Omega\setminus\{x_0\}\big).$$
Similarly, for any $x_0\in\Omega$ and any non-singular matrix $M\in\mathbb R^{n\times n}$, we define
 $$ A_{x_0,M}(y)=A(x_0+My), \ u_{x_0,M}(y)=M^{-1}\big(u(x_0+My)-u(x_0)\big),$$
for $y\in \Omega_{x_0, M}:=M^{-1}\big(\Omega\setminus\{x_0\}\big).$

The following scaling invariant property of absolute miminizers of $\mathscr F_\infty$
is a simple consequence of change of variables, whose proof is left for the readers.
\begin{lem}\label{l3.2} For any $x_0\in\Omega$, $r\not=0$, and a non-singular matrix
$M\in\mathbb R^{n\times n}$, if
$u\in C^{0,1}(\Omega)$ is an absolute minimizer of $\mathscr F_\infty$, with respect to $A$,
in $\Omega$, then $u_{x_0,r}$ is an absolute minimizer
of $\mathscr F_\infty$, with respect to $A_{x_0,r}$, in $\Omega_{x_0,r}$, and
$u_{x_0,M}$ is an absolute minimizer of $\mathscr F_\infty$, with respect to $A_{x_0,M}$,
 in $\Omega_{x_0, M}$.
\end{lem}

We also need the following lemma, which was proved in \cite{es11b}.
\begin{lem}\label{l3.3}
For ${\bf b}\in \mathbb S^{n-1}$ and $\eta>0$, if $v\in C^2(B(0,1))$ satisfies
$$\max_{x\in B(0,1)}\big|v(x)-\langle {\bf b},x\rangle\big|\le\eta,$$
then there exists a point $x_0\in B(0, 1)$ such that
$$\big|Dv(x_0)-{\bf b}\big|\le 4\eta.$$
\end{lem}

Now we are ready to prove Theorem 1.1.

\begin{proof}[Proof of Theorem 1.1]
For every point $x_0\in\boz$, we will show that there exists a vector $ Du(x_0)\in\mathbb R^n$
such that
\begin{equation}\label{e3.2}|u(x_0+ h)-u(x_0)-  \langle  Du(x_0),\,h\rangle|=o(|h|),
\ \forall\ h\in\mathbb R^n.
\end{equation}

From Lemma~\ref{l3.2},  we may assume that $x_0=0$, $u(x_0)=0$, and $A(x_0)=I_n$.
By Theorem \ref{t3.1}, in order to prove \eqref{e3.2},
it suffices to show that for every pair of  sequences ${\bf r}=\{r_j\}$ and ${\bf s}=\{s_k\}$ that converge  to $0$, if
\begin{equation}\label{e3.3}
\lim_{j\to\fz}\max_{y\in B(0,\,3r_j)}\frac 1 {r_{j }}\lf|  u( y)
-\langle {\bf a},\, y\rangle\r|=0
\end{equation}
and
\begin{equation}\label{e3.4}
\lim_{k\to\fz}\max_{y\in B(0,\,3s_k)}\frac1 {s_{k }}\lf| {u( y)}
-\langle {\bf b},\, y\rangle\r|=0
\end{equation}
for some ${\bf a},\ {\bf b}\in\rn$,
then ${\bf a}={\bf b}$.

Since $H(0, {\bf a})=\langle
{\bf a},\,{\bf a}\rangle=\langle
{\bf b},\,{\bf b}\rangle=H(0, {\bf b})={\rm{Lip}}_{d_A}u(0)$,  we have $|{\bf a}|=|{\bf b}|$.
We prove the above claim by contradiction. Suppose that $0\ne {\bf a}\ne {\bf b}$.
Then, without loss of generality, we may assume that ${\bf a}={e}_n$.
For, otherwise, let $M$ be a nonsingular matrix such that ${M{\bf a}=\bf e_n}$.
Set $v(y)=\frac{u( |{\bf a}| M^Ty)}{|{\bf a}|}$ and  $\wz A(y)={A(|{\bf a}|M^Ty)M}$.
Then {by Lemma~\ref{l3.2}} $v$ is an absolute minimizer of $\mathscr F_\infty$, with respect to $\wz A$.
It is clear that \eqref{e3.3} holds with $u$ and ${\bf a}$ replaced by $v$ and ${e}_n$ respectively.

Since  $|{\bf b} |=|{e}_n|=1$ and ${\bf b}\ne{e}_n$, we have
$$\theta:=1-b_n>0.$$
Let $C>0$ be the constant in (\ref{flat}) and choose
$\lz>0$ such that
$$C\lambda^\frac12=\frac{\theta}4.$$
Choose $r\in\{r_j\}$ such that
\begin{equation}\label{excess}
 \max_{y\in B(0,\,3r )}\frac 1 {r  }\lf|  u( y)
-y_n\r| \le\frac\lz 4,
\end{equation}
and
\begin{equation} \label{scaling1}
\begin{cases}
\frac2{1+2^{1/{5}}} |\xi|^2\le \big\langle A(x)\xi,\,\xi\big\rangle \le \frac{1+2^{1/{5}}}2|\xi|^2,
\ x\in B(0,3r), \ \xi\in\rn,\\
r\big\|  DA\big\|_{L^\infty(B(0,3))}+r^2\big\|   D^2A\big\|_{{L^{\infty}(B(0,3))}}\le \frac12\min\big\{\dz(B(0,\,3)), \lz \big\},
\end{cases}
\end{equation}
where $\dz(B(0,3))$ is the constant given by Theorem 3.2.

For $x\in B(0,3)$, let $\wz A(x)=A(r x)$ and $\wz u(x)=\frac1r u(rx)$. Since $D\wz A(x)=r (DA)(rx)$ and $D^2\wz A(x)=r^2 (D^2A)(rx)$ for $x\in B(0,3)$, it follows from (\ref{scaling1})  that
\begin{equation*}
\begin{cases}
\frac2{1+2^{1/{5}}} |\xi|^2\le \big\langle \wz A(x)\xi,\xi\big\rangle \le \frac{1+2^{1/{5}}}2|\xi|^2,
 \ x\in B(0,3),  \ \xi\in\rn,\\
\big\| D\wz A\big\|_{L^\infty(B(0,3))}
+\big\| D^2\wz A\big\|_{L^{\infty}(B(0,3))}\le \frac12\min\big\{\dz(B(0,\,3)), \lz \big\}.
\end{cases}
\end{equation*}
Let $\widetilde {A}_\epsilon\in\mathscr A(\Omega)\cap C^\infty(\Omega)$ such that
\begin{itemize}
\item[(i)]$\big\|\widetilde{A}_\epsilon\big\|_{C^{1,1}(B(0.3))}\le 2\big\|\widetilde A\big\|_{C^{1,1}(B(0,3))}$
for all $\epsilon>0$,
\item[(ii)]for any $0<\alpha<1$, $\widetilde{A}_\epsilon\rightarrow \widetilde A$ in $C^{1,\alpha}(B(0,3))$ as
$\epsilon\rightarrow 0$.
\end{itemize}
{Then there exists an $\ez_0>0$ such that for $\ez<\ez_0$
\begin{equation}\label{scaling3}
\begin{cases}
\frac2{1+2^{1/{4}}} |\xi|^2\le \big\langle \wz A_\ez(x)\xi,\xi\big\rangle \le \frac{1+2^{1/{4}}}2|\xi|^2,
 \ x\in B(0,3),  \ \xi\in\rn,\\
\big\| D\wz A_\ez\big\|_{{L^\infty(B(0,3))}}
+\big\| D^2\wz A_\ez\big\|_{{L^{\infty}(B(0,3))}}\le  \min\big\{\dz(B(0,\,3)), \lz \big\}.
\end{cases}
\end{equation}
 }
Let $\widetilde{u}^\epsilon\in C^{0,1}(B(0,3))$ be the unique solution of (\ref{e2.x2}) associated
with $\widetilde {A}_\epsilon$ and $H_{\widetilde{A}_\epsilon}$, with
$u$ and $\Omega$ replaced by $\widetilde u$ and $B(0,3)$ respectively.
Then, by Theorem 3.2,  we have that $\wz u^\ez\to \wz u$ uniformly in $B(0,3)$.
By Lemma \ref{l3.2}, $\wz u$ is an absolute minimizer of $\mathscr F_\infty$
with respect to $\wz A$.
From (\ref{excess}), we also have
\begin{equation*}
 \max_{y\in B(0,3)}\lf|  \wz u( y)- y_n\r| \le \frac\lz4.
\end{equation*}
Hence there exists {$\ez_1\in(0,\ez_0)$} such that for all $\ez<\ez_1$, 
\begin{equation}\label{flat_ass}
 \max_{y\in B(0, 3)}\lf|  \wz u^\ez( y)
 - y_n\r| \le \frac\lz2.
\end{equation}
Setting $\wz s_k=s_k/r$. Then we have
\begin{equation*}
\lim_{k\to\fz}\max_{y\in B(0,3\wz {s_k} )} \frac 1{\wz s_{k }}\lf|{\wz u( y)}
-\langle {\bf b},\, y\rangle\r|=0.
\end{equation*}
Choose $\eta=\frac{\theta}{48}$ and pick $s\in\{\wz s_k\}$, with $0<s<1$, so that
\begin{equation*}
 \max_{y\in B(0,\, s  )} \frac 1{s}\lf|{\wz u( y)}
-\langle {\bf b},\, y\rangle\r|\le \frac\eta2.
\end{equation*}
By Theorem 3.2, there exists $\ez_2>0$ such that for all $\ez<\ez_2$,
\begin{equation*}
 \max_{y\in B(0,\, s)}{\frac1s}\lf|  \wz u^\ez( y)
 -\langle {\bf b},\, y\rangle\r| \le\eta.
\end{equation*}
{Applying Lemma \ref{l3.3} to   $\frac1s\wz u^\ez(s\cdot)$,  we can find} 
 a point $x_0\in B(0, s)$ such that
$$
\lf|  D\wz u^\ez( x_0)
 -  {\bf b} \r| \le 4\eta,$$
which,  combined with $|{\bf b}|=1$,  yields
 \begin{equation}\label{low_upp}
\begin{cases}
\wz u^\ez_n( x_0)\le b _n + 4\eta\le 1-\theta+4\eta,\\
\lf|  D\wz u^\ez( x_0) \r|\ge 1-
  4\eta.
\end{cases}
\end{equation}
From (\ref{flat_ass}), we can apply Theorem 3.3 to conclude
  $$\lf|  D\wz u^\ez( x_0) \r|^2\le \wz u^\ez_n(x_0)+C \lz^{1/2}
\le \widetilde{u}^\epsilon_n(x_0)+\frac{\theta}4.$$
This, combined with (\ref{low_upp}), implies that
 $$(1-4\eta)^2\le 1-\theta+4\eta+  \frac{\theta}4,  $$
so that
  $$\theta\le 12 \eta+ \frac{\theta}4 \le\frac\theta 2, $$
this is impossible. Thus ${\bf a}={\bf b}$, and there is a unique tangent
plane at $0$ and $u$ is differentiable at $0$. The proof is complete.
\end{proof}

\section{Lebesgue points of the gradient}

In this last section, we show {that every point is a Lebesgue point for the gradient}, which extends the property
on infinity harmonic functions by \cite{es11a}.

\begin{thm}
Let $A\in\mathscr A(\Omega)\cap C^{1,1}(\Omega)$ and $u$ be a viscosity solution of the Aronsson equation \eqref{e1.x1}.
Then every point in $\boz$ is a Lebesgue point of $Du$.
\end{thm}

For the intrinsic distance $d_A$ associated  with $A$, define the intrinsic ball
$$B_{d_A}(x,r):=\Big\{y\ \big|\ d_A(x,y)<r\Big\}$$
for $x\in\Omega$ and $0<r<d_A(x,\partial\Omega)$. For $E\subset\mathbb R^n$,
define $\displaystyle\bint_E f=\frac{1}{|E|}\int_E f.$

\begin{lem}
For $0<\lambda<1$,  let $A\in\mathscr A(\Omega)\cap C^{1,1}(\Omega)$ such that
$A(0)=I_n$ and $\big\|DA\big\|_{L^\infty(\Omega)}\le \lz^2$.
Assume $u\in C^{0,1}(\Omega)$ is an absolute minimizer of $\mathscr {F}_\infty$
with respect to $A$, and satisfies, for $B_{d_A}(0,3)\subset\Omega$,
$$\max_{x\in B_{d_A}(0,3)}\big|u(x)-u(0)-\langle {\bf a},x\rangle\big|\le \lz.$$
Then there exists a constant $C>0$ depending on $|{\bf a }|$ such that
\begin{equation}\label{lebesgue}
\bint_{B_{d_A}(0,1)}
\big|Du(x)- {\bf a} \big|^2\,dx\le C\lz.
\end{equation}

\end{lem}

\begin{proof}
Since
$$(1+C\lz^2)^{-1}|\xi|^2\le  \langle A(x)\xi,\xi\rangle \le (1+C\lz^2)|\xi|^2,
\ \forall x\in\Omega, \ \xi\in\mathbb R^n,$$
we have
$$(1+C\lz )^{-1}|x-y|\le d_{A}(x,\,y)\le (1+C\lz )|x-y|, \ \forall x, y\in\Omega.$$
It  suffices to show that
\begin{equation}\label{lebesgue1}
\bint_{B (0,1+C\lambda)}
\big|Du(x)- {\bf a}\big|^2\,dx\le C\lz.
\end{equation}
By the same argument as in the proof of \cite{es11a} Theorem 4.1,
(\ref{lebesgue1})  follows if
\begin{equation}
\sup_{x\in B(0,1+C\lambda)} |Du(x)|\le |{\bf a}|+C\lz.
\label{grad_bd1}
\end{equation}
 To prove (\ref{grad_bd1}), let
$$S^+_ru(x):=\max_{d_{A}(z,x)=r} \frac{u(z)-u(x)}{r}.$$
A simple modification of the proof of \cite{ksz} Lemma 2.2
shows that $\displaystyle S^+_ru(x)$ is monotone increasing with respect to $r$,
and $$\sqrt{\langle A(x)Du(x),Du(x)\rangle}=\lip_{d_A}u(x)=\lim_{r\to0}S^+_ru(x).$$
This implies
$$|Du(x)|\le (1+C\lz) S^+_1u(x), \ x \in B(0,1+C\lambda).$$
For $x\in B(0, 1+C\lambda)$, if $B_{d_A}(x,1)\subset B_{d_A}(0,3)$  and $d_A(z,x)=1$,
then we have
\begin{eqnarray*}|u(x)-u(z)|&\le&|u(x) -u(0)-\langle{\bf a},x\rangle|+|u(z) {-u(0)}-\langle{\bf a},z\rangle|+
| \langle{\bf a},x-z\rangle|\\\
&&\le 2\lz+|  {\bf a }||x-z|\le |  {\bf a} | +C\lz,
\end{eqnarray*}
which implies that
$$S^+_1u(x)\le |  {\bf a} | +C\lz, \ \forall x\in B(0,1+C\lambda). $$
Hence we have that
$$|Du(x)|\le |  {\bf a} | +C\lz, \ \forall\ x \in B(0, 1+C\lz).$$
{The proof is completed by applying the argument in Theorem 4.1 of~\cite{es11a}.}
\end{proof}

\begin{proof}[Proof of Theorem 5.1]
We want to show that for every $x_0\in\boz$ and for every $\ez>0$,
there exists $r_0>0$ such that
$$\bint_{B_{d_A}(x_0,r)} \big|Du(x)-Du(x_0)\big|^2 \,dx\le \epsilon.$$
{for every $r \le r_0$}.
{As before,} by Lemma 4.2, we may assume that $x_0=0$, $u( 0)=0$ and $A(0)=I_n$.
For an arbitrary $0<\lz<1$, since $u$ is differentiable at $0$,
there exists $r_0<\lz^2$ such that
 \begin{equation}\label{excess1}
\max_{z\in B_{d_A}(0,\,3r)}\big|u(x)- \langle Du(0),x\rangle\big|\le \lz r, \ 0<r\le r_0.
\end{equation}
Set $A_r(x)=A(rx)$ and $\displaystyle u_r(x)=\frac{u(rx)}{r}$. Then $u_r$ is an absolute  minimizer
of $\mathscr{F}_\infty$ associated to $A_r$.
Observe that $\displaystyle\|DA_r\|\le r\|DA\|$ and by \cite{ksz} Lemma 5.4,
$d_{A_r}(rx,ry)=rd_A(x,y)$. Hence $B_{d_A}(0,r)=rB_{d_{A_r}}(0,1)$.
Therefore,  (\ref{excess1}) implies
\begin{equation}
\label{excess2}
\max_{x\in B_{d_{A_r}}(0,\,3 )}\big|u_r(x)- \langle Du(0),x\rangle\big|\le \lz,
\ 0<r\le r_0.
\end{equation}
Now we can apply Lemma 5.2 to conclude that
\begin{eqnarray*}\bint_{B_{d_A}(0,r)}\big|Du(x)-Du( 0)\big|^2\,dx
&&
=\bint_{B_{d_{A_r}}(0,1)}\big|Du_r(x)-Du( 0)\big|^2\,dx\\
&&\le C\lz,
\end{eqnarray*}
{for every $r \le r_0$ and $\lambda$ small enough.} This completes the proof.
\end{proof}

\bigskip
\noindent{\bf Acknowledgement}.
J. Siljander was supported by the Academy of Finland
grant 259363 and a V\"ais\"al\"a foundation travel grant. C . Wang  was partially supported by NSF grants DMS 1265574 and DMS 1001115, and NSFC grant 11128102.
Y. Zhou was supported by
the New Teachers' Fund for Doctor Stations (\# 20121102120031) and Program for New Century Excellent Talents in University (\# NCET-11-0782)  of Ministry of
Education of China, and National Natural Science Foundation of China (\# 11201015).

\noindent Juhana Siljander,

\noindent Department of Mathematics,   University of Helsinki,
Finland
\smallskip

\noindent{\it E-mail}:   \texttt{juhana.siljander@helsinki.fi }

\bigskip

\noindent Changyou Wang

\noindent  Departmet of Mathematics, Purdue University,
150 N. University Street, West Lafayette, IN 47907, USA.

\smallskip

\noindent {\it E-mail}: \texttt{wang2482@purdue.edu }  

\bigskip

\noindent Yuan Zhou

\noindent
Department of Mathematics, Beijing University of Aeronautics and Astronautics, Beijing 100191, P. R. China

\noindent{\it E-mail }:  \texttt{yuanzhou@buaa.edu.cn}


\end{document}